\documentclass[12pt]{article}
\usepackage{amsmath,amsfonts,amssymb,amsthm,url,color,epsfig,graphics}
\usepackage[colorlinks=true]{hyperref}

\def\R{\mathrm{I\kern-0.21emR}}
\def\N{\mathrm{I\kern-0.21emN}}
\newcommand{\Z}{{\mathbb{Z}}}

\newcommand{\D}{{\mathbb{D}}}
\newcommand{\sphere}{{\mathbb{S}}}

\newcommand{\ha}{{\frac{1}{2}}}

\renewcommand{\geq}{\geqslant}
\renewcommand{\leq}{\leqslant}

\newtheorem{definition}{Definition}[section]

\newtheorem{remark}{Remark}

    \newtheorem{theorem}{Theorem}[section]
    \newtheorem{lemma}{Lemma}[section]
    \newtheorem{proposition}{Proposition}[section]
    
    \newtheorem{corollary}{Corollary}[section]

    \let\eps\varepsilon
    \let\epsilon\varepsilon

    \let\oldsum\sum
    \renewcommand{\sum}{\displaystyle\oldsum}
    \let\oldprod\prod
    \renewcommand{\prod}{\displaystyle\oldprod}
    \let\oldinf\inf
    \renewcommand{\inf}{\displaystyle\oldinf}
    \let\oldsup\sup
    \renewcommand{\sup}{\displaystyle\oldsup}

    \let\leq\leqslant
    \let\geq\geqslant
    \newlength{\oldparindent}
    \newcommand{\myindent}{\hspace{\oldparindent}}

    \title{On the $L^\infty $ norms of spectral projectors on shrinking intervals: the cases of some spheres of revolution and of the Euclidean disk }
    
    \author{Ambre Chabert\footnote{Département de Mathématiques et Applications, Ecole Normale Supérieure,
        UMR 8553, 45 rue d'Ulm. 75230 Paris Cedex 05, France. Email adress : ambre.chabert@ens.fr}~ \&  Yves  Colin de Verdière\footnote{
    Institut Fourier, Université Grenoble-Alpes,
          UMR 5582,
      BP 74, 38402-Saint Martin d'Hères Cedex, France.
    Email adress : yves.colin-de-verdiere@univ-grenoble-alpes.fr}}

\begin{document}
    \maketitle
    \begin{abstract}
        Given a compact Riemannian surface $M$, with Laplace-Beltrami operator $\Delta$, for $\lambda > 0$, let $P_{\lambda,\lambda^{-\frac{1}{3}}}$ be the spectral projector on the bandwidth $[\lambda-\lambda^{-\frac{1}{3}}, \lambda + \lambda^{\frac{1}{3}}]$ associated to $\sqrt{-\Delta}$. We prove a polynomial improvement on the $L^2 \to L^{\infty}$ norm of $P_{\lambda,\lambda^{-\frac{1}{3}}}$ for generic simple spheres of revolution (away from the poles and the equator) and for the Euclidean disk away from its center, but up to the boundary. We use the Quantum Integrability of those surfaces to express the norm in terms of a joint basis of eigenfunctions for $\left(\sqrt{-\Delta}, \frac{1}{i}\frac{\partial}{\partial \theta}\right)$. Then, we use that those eigenfunctions are asymptotically Lagrangian oscillatory functions, each supported on a Lagrangian torus with fold-type caustic. Thus, studying the distribution of the caustics, and using BKW decay away from the caustics, we are able to reduce the problem to counting estimates.
    \end{abstract}
    \section{Introduction}
    
    \myindent Let $(M,g)$ be a smooth compact Riemannian surface, with Laplace-Beltrami operator $\Delta$. If $\partial M$ is non empty, we fix Dirichlet or Neumann boundary conditions. We introduce, in the spirit of Sogge (see \cite{sogge1988concerning}) the associated spectral projectors on thin frequency intervals.

    \begin{definition}\label{defspecproj}
        Let $\lambda,\delta > 0$, we define through functional calculus
        \[
            P_{\lambda,\delta} := 1_{[\lambda - \delta, \lambda + \delta]}(\sqrt{-\Delta}),
        \]
        where $1_{[\lambda - \delta, \lambda + \delta]}$ is the indicator function of $[\lambda - \delta, \lambda + \delta]$
    \end{definition}

    \myindent The question of estimating the operator norm of $P_{\lambda,\delta}$, seen as an operator from $L^2(M)$ to $L^{\infty}(M)$,
    in relation with the geometry of $M$, has become in the last decades an active problem in Spectral Geometry, see the review \cite{germain2023l2}. In particular, observe that any bound on $\|P_{\lambda,\delta}\|_{L^2\to L^{\infty}}$ implies a corresponding bound on the $L^{\infty}$ norm of eigenfuntions compared to the general upper bound (see below). Indeed, if $\phi_\lambda$ is an $L^2$ normalized eigenfunction of $\Delta$, i.e.
    \[\Delta \phi_\lambda = - \lambda^2 \phi_\lambda,\]
    we have, for any compact $K\subset M$, $\|\phi_\lambda\|_{L^{\infty}} \leq \|P_{\lambda,\delta}\|_{L^2(M) \to L^{\infty}(K)}$.
    
    \myindent In this article,
    we study the spectral projectors on polynomially thin frequency intervals for some examples  of \textit{Quantum  integrable} surfaces as presented in Section \ref{sec:qi}. Our two main models are the generic simple spheres of revolution, and the Euclidean disk.
    \subsection{Results for the two model cases : the spheres of revolution and the Euclidean disk}
    
    \begin{definition}\label{def:si} 
      A metric of revolution on the 2-sphere $\sphere$ is said to be {\rm simple} if the metric is smooth and writes outside the poles
      $g= a(s)^2 d \theta ^2 + ds ^2  $ with $a: [0, L] \rightarrow  [0,\infty  [ $ a smooth function vanishing at the boundaries and
          with only one critical point which is a non degenerate maximum.
\end{definition}

    \myindent For a \textit{generic} simple metric of revolution on the sphere, it was proved in \cite{de1980spectre} that, \textit{generically}, the remainder of the Weyl counting function is better than in the general case, i.e. if $0\leq \lambda_1^2 \leq \lambda_2^2 ...$ are the eigenfunctions of $-\Delta$, and $N(\lambda) = \# \{j\geq 1: \ \lambda_j \leq \lambda\}$, then 
\begin{equation}\label{generic}
    N(\lambda) = c\lambda^2 + O(\lambda^{2/3}).
\end{equation}

\myindent In the following, we will simply call a sphere of revolution \textit{generic} if it satisfies estimate \eqref{generic}. We then prove the following theorem on generic simple spheres of revolution away from the poles and the equator.

    \begin{theorem}\label{main}\footnote{
         We recall the  notations 
        \[ 
        A\lesssim_{a,b,c,...} B
        \]
        if and only if there exists a constant $C$, which may depend on the indices $a,b,c,...$, such that
        \[
            A \leq C B.
        \]
        and $A \simeq _{a,b,c,...} B$ if and only if $A \lesssim_{a,b,c,...} B$ and  $B \lesssim_{a,b,c,...} A$.
    - }
      Assume that $g$ is a simple (definition \ref{def:si})  and generic (equation \eqref{generic}) metric of revolution on $\sphere$.
      Let $U:=\sphere \setminus \left( {\rm poles}\cup {\rm equator}\right)$ and 
             $K\subset U$ be compact. There holds
            \begin{equation} \label{specprojupbdd}
            \|P_{\lambda,\lambda^{-\frac{1}{3}}}\|_{L^2(\sphere) \to L^{\infty}(K)} \lesssim_{K} \lambda^{\frac{1}{2} - \frac{1}{12}} \qquad \forall \lambda > 0.
            \end{equation}
    \end{theorem}
    
    \myindent The other main example of Quantum Integrable surface which we study is the \textit{Euclidean disk} $\mathbb{D}^2 \subset \R^2$, i.e. the closed disk of radius $1$ centered at the origin $O$. It was proved that the disk satisfies the remainder estimate \eqref{generic} in \cite{YCdV-disque,kuznetsov1965asymptotic} In this context, the knowledge of an explicit basis of eigenfunctions in terms of Bessel functions yields that we may prove a similar theorem, and even go up to $\partial \mathbb{D}^2$ in the estimate.
    
    \begin{theorem}\label{main2}
      \begin{itemize}
        \item
            If $K\subset {\rm Int}(\D)\setminus O$ is compact, there holds, for $\lambda \rightarrow + \infty $,
            \begin{equation} \label{specprojupbdd2}
            \|P_{\lambda,\lambda^{-\frac{1}{3}}}\|_{L^2(M) \to L^{\infty}(K)} \lesssim_{K} \lambda^{\frac{1}{2} - \frac{1}{12}}.
            \end{equation}

            \item If $K$ is a compact set with $K\subset \D \setminus O$ ($K$ possibly intersects the boundary), there holds, for $\lambda \rightarrow + \infty $,
            \begin{equation}\label{specprojupbdd3}
            \|P_{\lambda,\lambda^{-\frac{1}{3}}}\|_{L^2(M) \to L^{\infty}(K)} \lesssim_{K} \lambda^{\frac{1}{2} - \frac{1}{18}} .
            \end{equation}

            \end{itemize}
    \end{theorem}
    
    \subsection{Quantum integrability of geodesic flows.}\label{sec:qi} 
    
    \myindent The two models which we introduced are some of the main examples of the more general class of \textit{Quantum integrable surfaces}. Let $(M,g)$ be a compact  Riemannian surface. Let $\Delta $ be the Laplace-Beltrami operator, let $g^\star $  be the dual metric,
    which turns out to be the principal symbol of $-\Delta $, and let $P_1=\sqrt{-\Delta }$, whose principal symbol is
    $p_1=\sqrt{g^\star}$.
    \begin{definition}
    We say that the Laplace-Beltrami operator $\Delta $ is {\rm Quantum integrable} (QI) if there exists a self-adjoint pseudo-differential operator (a $\Psi$DO)
    $P_2$ on $M$ of degree $1$ commuting with $\Delta $ such that the differential $dp_2 $ of the principal symbol $p_2$ of $P_2$  is linearly independent of
    $dp_1 $ on an open conic dense set of $T^\star M \setminus 0$.
\end{definition}
    \myindent This situation is discussed in \cite{YCdV-sp-joint-II}, and many examples of QI systems are given in the more recent \cite{Bam25}. For surfaces of revolution, as well as for the Euclidean disk, the operator $P_2 := \frac{1}{i}\frac{\partial}{\partial \theta}$ commutes with the Laplace-Beltrami operator, and hence with $P_1 = \sqrt{-\Delta}$.
    
    \myindent Now, observe that being QI is stronger than the classical integrability of the geodesic flow:
    if $\Delta $ is QI,  the Hamiltonian $p_1$ of the geodesic flow is indeed classically integrable
    with first integral $p_2$. In particular, the standard theory of completely integrable homogeneous hamiltonian flows yields that the phase space $T^*M \backslash 0$ is "foliated" by level sets of $p := (p_1,p_2)$ which are tori of dimension $1$ or $2$ or cylinders. For simple spheres of revolution, and for the Euclidean disk, this foliation has a simple structure. Indeed, in those cases, there is a codimension $2$ subset of $T^*M\backslash 0$ on which the level sets of $p$ form a smooth foliation by lagrangian tori, which projects "nicely" onto $M$ in the sense of the following definitions.
    
\begin{definition}
    Let $L \subset T^*M\backslash 0$ be a compact lagrangian submanifold, let $\Pi : T^*M \to M$ be the canonical projection, and let $\pi : L\to M$ be its restriction to $L$. Following \cite{Gu-Sc}, we say that $p\in L$ is singular if $d\pi(p)$ has a nonzero kernel, and we define $S\subset L$ the set of singular point. We recall that the {\rm caustic set} of $L$ is defined as $\pi(S)$. 
    
   \myindent Assume that $S$ is a submanifold of $L$ of dimension $1$. We say that $p\in S$ is of {\rm fold-type} if $ker(d\pi(p))$ and $T_p S$ span together $T_p L$. 
\end{definition}

\myindent We then define the subsets of $M$ over which the phase-space is foliated by lagrangian tori with fold-type singularities.
    
 \begin{definition}\label{def:regdom}
    The {\rm regular phase space  domain} is the  conic open  subset $\Omega $ of  $T^\star M \setminus 0$  where the  ``foliation'' defined by the level sets of
    the joint symbol $p=(p_1,p_2 )$ is a foliation by smooth Lagrangian tori.

    \myindent Let $\Pi $ be the canonical projection of $T^\star M $ onto $M$. A {\rm regular domain} $U$  in $M$ is an open set $U$ so that $\Pi^{-1}(U) $ is contained in the  regular phase space domain.

   \myindent  A regular  domain $U$  is said to be of {\rm  fold type} if the singularities of the Lagrangian tori are of fold type over $U$.
\end{definition}
\begin{remark}
      Note that the definition of a  regular domain forbids the existence of an unstable hyperbolic closed geodesic $\gamma $ in $M$.
      Indeed, in the later case, the stable and
      unstable manifolds of $\gamma $ project onto the full manifold.
\end{remark}

    \myindent For $f$ a \textit{simple} metric of revolution on $\sphere$, the regular phase space domain is 
    \[\Omega =\{ (s,\theta;\sigma, p_2 )~|~|p_2 |< A p_1 \}\] with $p_1= \sqrt{\sigma^2+ \frac{p_2 ^2}{a(s)^2 }}$
and  $A= 1/ \sup_{s\in [0,L]} a(s) $. The Lagrangian tori are given by $L_{u,v} = p^{-1}\{(u,v)\}$, for  $(u,v) \in \R^2$ fixed such that $0 < |v| < Au$.

\myindent The open set $U:=M\setminus ~({\rm poles}~\cup ~{\rm equator}) $ is a maximal regular domain of fold type: 
the Lagrangian manifolds $(p_1=\mu, p_2=0)$ admit the poles as an unique caustic point which is indeed not of fold type; moreover   
the Lagrangian foliation is singular at the equatorial geodesic. We will keep this $U$ in what follows concerning the spheres of revolution. 

\myindent Similarly, for $\mathbb{D}^2$, the same formalism applies, with maximal regular domain of fold-type given by $U:= \mathbb{D}^2 \backslash (\{O\} \cup \partial\mathbb{D}^2)$.

\quad

\myindent Now, for more general QI surfaces, there is usually no regular domain, since there typically is at least one hyperbolic geodesic in $M$. However, the analysis may still be developed, after microlocalizing on that part of the regular phase space which is foliated by Lagrangian tori of fold type. We enunciate the result in this case as a remark, since, for the sake of concreteness, we will not give a full proof of it in the present article, however we will indicate which points needs to be changed in order to prove this result.

\begin{remark}\label{rem:gen}
    Let $(M,g)$ be a Quantum Integrable surface, and let $\tilde{\Omega} \subset \subset \Omega$ be an open subset of the regular phase space domain which is foliated by smooth Lagrangian tori of fold-type. Let $\Psi(x,D)$ be a pseudodifferential operator of order zero whose wavefront set is supported in $\tilde{\Omega}$. Then, provided the spectrum of $(M,g)$ satisfies the remainder estimate \eqref{generic}, there holds
    \[\left\| P_{\lambda,\lambda^{-1/3}} \Psi(x,D)\right\|_{L^2(M)\to L^{\infty}(M)} \lesssim_{\Psi} \lambda^{1/2 - 1/12}.\]
\end{remark}

    
    \subsection{Earlier results}

     \myindent In the seminal paper \cite{hormander1968spectral}, Hörmander proved, using microlocal analysis, that, for a general Riemannian compact manifold $M$ of dimension $d$, and for any $L^2$ normalized eigenfunction $\phi_{\lambda}$ of the Laplacian (i.e. $-\Delta \phi_{\lambda} = -\lambda^2 \phi_{\lambda}$), there holds the following polynomial upper bound on its $L^{\infty}$ norm
    \begin{equation}\label{hormbd}
        \|\phi_{\lambda}\|_{L^{\infty}(M)} \lesssim_{M} \lambda^{\frac{d-1}{2}}.
    \end{equation}

     \myindent While this upper bound is sharp in general, since it is saturated by zonal spherical harmonics on the $d$-dimensional sphere $S^d$, it is expected that, generically, one can prove much better bounds. Indeed, it was proved in \cite{sogge2002riemannian} that, for a \textit{generic} $M$, there holds
    \[
        \|\phi_{\lambda}\|_{L^{\infty}(M)} = o\left(\lambda^{\frac{d-1}{2}}\right).
    \]

     \myindent With additional assumption on the geometry of $M$, this upper bound can be \textit{quantitatively} improved. 

     \myindent For the regular flat torus $\mathbb{T}^d$, $d > 2$, Bourgain conjectured in \cite{bourgain1993eigenfunction} that there holds
    \[
        \|\phi_{\lambda}\|_{L^p(\mathbb{T}^d)} \lesssim \lambda^{\frac{d}{2} - 1 - \frac{d}{p}} \qquad p > \frac{2d}{d-2}.
    \]

     \myindent Some cases of this conjecture were proved in \cite{bourgain2013moment}, culminating with the $l^2$ decoupling theorem, see \cite{bourgain2015proof}. More generally, estimates on spectral projectors on thin frequency intervals for tori have been extensively studied in \cite{germain2022boundsBook,germain2022bounds,demeter2024l2,hickman2020uniform}.

     \myindent For \textit{arithmetic surfaces}, it is conjectured in \cite{iwaniec1995norms} that there holds $\|\phi_{\lambda}\|_{L^{\infty}}\lesssim \lambda^{\eps}$ for any $\eps > 0$. For recent progresses on that conjecture, see \cite{buttcane2017fourth,humphries2018equidistribution,humphries2022p}.

     \myindent More generally, for manifolds with \textit{nonpositive curvature}, many works have obtained \textit{logarithmic} improvements on the norm of eigenfunctions, see \cite{berard1977wave,hassell2015improvement,hezari2016lp,blair2017refined,blair2018concerning,blair2019logarithmic}.

     \myindent The converse question of finding geometric assumption on $M$ under which there exists sequences of eigenfunctions with high $L^{\infty}$ (or $L^p$) norms has been studied in many works, such as \cite{sogge2002riemannian,sogge2011blowup,sogge2016focal,sogge2016focalb,sogge2001riemannian,canzani2019growth,canzani2021eigenfunction}.
     
     \myindent We mention the setting of the \textit{magnetic Laplacian} on compact hyperbolic surface, for which one observes three of the main behaviors, depending on energy levels \cite{Chabert26}: for low energy levels, the situation is similar to the sphere, that is Hörmander's bound is saturated. For the critical energy level, a polynomial improvement holds. Finally, for high energy levels, the situation is similar to manifolds with nonpositive curvature.

    \quad

    \myindent In the context of \textit{completely integrable} manifolds, it is expected that, \textit{generically}, there should be polynomial improvements on the upper bound \eqref{hormbd}. Indeed, Bourgain claimed in \cite{bourgain1993eigenfunctiona} that, for a generic subset $U$ of a generic completely integrable manifold $M$, there holds
    \[
     \|\phi_{\lambda}\|_{L^{\infty}(U)} \lesssim \lambda^{\frac{d-1}{2}-
                                    \eps},
    \]
    where $\eps > 0$ depends on $M$. The converse question of finding sequences of eigenfunctions with high $L^p$ norms has been studied in \cite{toth2003lp,toth2003norms,toth2002riemannian}.

     \myindent For \textit{quantum completely integrable} manifolds (see
    \cite{eswarathasan2024pointwise} for a definition), introduced by one
    of the authors in \cite{de1980spectre}, the other author conjectured
    in \cite{chabert2025bounds} a quantitative approach to Bourgain's
    conjecture. Since the latest paper is closely related to this one, we
    briefly discuss its results compared to ours.

     \myindent In \cite{chabert2025bounds}, the first author studied a generic simple surface of revolution $\mathcal{S}$ with an additional convexity assumption. In particular, her analysis doesn't apply to the Euclidean disk. She proved that, if $K$ is any compact subset of $\mathcal{S}$ which doesn't contain a pole, then there holds, with the Definition \ref{defspecproj},
    \begin{equation}\label{exactest}
        \|P_{\lambda,\delta}\|_{L^2(\mathcal{S}) \to L^{\infty}(K)} \simeq_{\mathcal{S},K} \lambda^{\frac{1}{2}}\delta^{\frac{1}{2}} \qquad \forall \delta > \lambda^{-\frac{1}{32}}.
    \end{equation}

     \myindent In particular, this yields a $\lambda^{\frac{1}{2} - \frac{1}{64}}$ upper bound on the $L^{\infty}$ norm of eigenfunctions away from the poles. Hence, in this paper, we improve this bound to a $\lambda^{\frac{1}{2} - \frac{1}{12}}$ upper bound away from the poles \textit{and} the equator, under less restrictive assumptions. However, we conjecture that the upper bound \eqref{specprojupbdd} on the spectral projector is \textit{not} optimal, and that, rather, the identity \eqref{exactest} holds up to $\delta = \lambda^{-\frac{1}{3}}$, see \cite{chabert2025bounds}[Conjecture 8.1]. We insist that, due to this non optimality, the upper bound \eqref{specprojupbdd} \textit{doesn't} imply the upper bound in \eqref{exactest}. Hence, in the context of spheres of revolution, while this paper yields a more direct approach to improved upper bounds on eigenfunctions, the approach of the aforementioned paper is probably more suited in order to study spectral projectors on thin frequency intervals. 
    
    \subsection{Ingredients and ideas for  the proof}
    
    \myindent The first part of the proof is to decompose on the basis of joint eigenfunctions of $P_1 = \sqrt{-\Delta}$ and $P_2 = -i\frac{\partial}{\partial \theta}$. This basis is very well controlled since
    
    \begin{enumerate}
        \item The joint eigenfunctions "live" in the regular phase space domain $\Omega$, and they are Lagrangian oscillatory functions (in the sense of \cite{duistermaat1974oscillatory}, see Appendix \ref{app:lagr}). This is explicit for the disk from the expression of Bessel functions.
        
        \item The fold singularity of Lagrangian tori implies that the eigenfunctions are well-approximated in terms of Airy functions (see \ref{sec:caustics}), which in turn yields a BKW decay away from the caustics (see Proposition \ref{airyestimate}).
        
        \item The caustic sets associated to the Lagrangian foliation defines a foliation of $U$ (section \ref{sec:fol})
        
        \item The joint spectrum\footnote{For the general case mentioned in Remark \ref{rem:gen}, the joint spectrum is well expressed with the notion of \textit{deformed lattice},see \cite{de1977quasi}, Theorem 5.5,   or for a semi-classical version \cite{San06},
Theorem 4.1.11.} is well expressed in terms of the integer lattice $\Z^2$. 
    \end{enumerate}
    
    \myindent Let us discuss how we will use these ingredients in the case of $\sphere$. The case of the Euclidean disk is similar away of the boundary, but we use directly
    the known asymptotics of the eigenfunctions which turn out to express in terms of Bessel functions. Near the boundary, we rely on the same idea but we have to refine the estimate of gaps between caustic sets, and on more careful counting estimates.

    \myindent From the first two ingredients,  we know the asymptotic behaviour of the eigenfunctions near the caustics. Using the two next ingredients, we can estimate the distances between caustics for eigenfunctions corresponding to eigenvalues of $\Delta $ in the
    small interval.
    Finally,  we use  a simple counting argument using the improved remainder estimate in the Weyl law \eqref{generic}. 

    \myindent More precisely, let us start with the formula
    \begin{equation} \label{equ:id} \|  F(\sqrt{-\Delta})\|_{L^2(M) \rightarrow L^\infty(K)}^2  = \sup _{x\in  
    K} \sum _{j}F^2\left(\lambda_j\right)
   |\phi _j (x)|^2 \end{equation} 
    where $\lambda _j$'s are the eigenvalues of $\sqrt{-\Delta} $, and $\phi_j $ are the $L^2$-normalized joint eigenfunctions. Then, for a fixed $x \in K$, we may use the BKW decay estimate to bound $|\phi_j(x)|$ in terms of the distance between $x$ and the caustic of the Lagrangian torus associated to $\phi_j$. Now, we conclude by carefully using the distribution of the caustic sets.

    \myindent Note that we will use the {\it semi-classical approach} using $h:=(-\lambda )^{-\ha } $ as  semi-classical parameter. Then the joint spectrum system of equations writes
    \[ h P_1\phi =\phi, hP_2\phi =\mu  \phi \]
    where $\mu = \lim _{\lambda  \rightarrow  \infty } \frac{n  }{ \lambda}$ with $(\lambda, n )$ in the joint spectrum
    of $P_1 $ and $P_2 $.
    
      
    \section{Generic simple spheres of revolution}
    \subsection{First results}
    \subsubsection{The joint spectrum}
        \myindent We recall the following result, which follows from \cite{de1980spectre}[Section 6] and \cite{chabert2025bounds}[Lemma 2.8, eq. (3.7)].
        \begin{proposition}
            There exists $F \in \mathcal{C}^{\infty}(\R^2,\R_+)$ homogeneous of degree 1 such that the eigenvalues of $\Delta$ are given by
            \[\lambda_{k,n} := F(k+1/2, n) + O(|(k,n)|^{-1/2}) \qquad \ n\in \Z, \ k \geq |n|.\]
            
            \myindent Moreover, there holds 
            \[\partial_1 F \geq c > 0\]
            for some constant $c$. 
            
            \myindent Finally, if we let $Z := \{(\lambda_{k,n}, n): \ n\in \Z, \ k \geq |n|\}$ be the joint spectrum of $P_1$ and $P_2$, then all eigenvalues are simple, and we can thus find a corresponding basis of eigenfunctions
            \[\phi_{k,n}(s,\theta) := \psi_{k,n}(s) e^{in\theta}.\]
        \end{proposition}
    
    \myindent As a corollary, we may deduce the following gap lemma.
    \begin{lemma}\label{gaplem}
        For $\eps > 0$ fixed and $\lambda > 0$, let $B_{\lambda} := \{(u,v) \in \R^2: \ |u-\lambda|\leq \lambda^{-1/3}, \ |v| \leq u\}$. Then, for $\lambda$ large enough, for some constant $c > 0$, for any distinct pairs $\nu_j = (\lambda_{k_j,n_j},n_j) \in Z \cap B_\lambda$ $(j=1,2)$, there holds $|n_1 - n_2| \geq 1$.
    \end{lemma}
    \begin{proof}
        Assume that $n_1 = n_2 = n$. Then, there must hold $k_1 \neq k_2$, hence $|k_1 - k_2| \geq 1$ since they are integers. In particular, from the proposition,
        \[|\lambda_{k_1,n} - \lambda_{k_2,n}| \geq c - |(k_1,n)|^{-1/2}- |(k_2,n)|^{-1/2} \geq c/2\]
        as long as $\lambda$ is large enough. Thus it is not possible that $\nu_1$ and $\nu_2$ are both in $B_\lambda$ if (say) $\lambda^{-1/3} < c/2$.
    \end{proof}
    
    \subsubsection{The joint eigenfunctions are Lagrangian}
    
    \myindent Let $D$ be the closed subcone of $\R^2$ which is the image of the joint symbol $p=(p_1,p_2)$, i.e. $D = \{m=(u,v) \ | |v| \leq Au\}$ with $A = \sup_{s\in [0,L]} a(s)$. Let $C$ by the interior of $D$. We recall the following fact (see \cite{de1980spectre})
    \begin{lemma}\label{prop:Z}
      The set $\Omega := p^{-1}(C)$ is the regular phase space domain. In particular, for any $m\in C$, $L_m = p^{-1}(m)$ is a Lagrangian torus.
    \end{lemma}
    
    \myindent As a consequence, we obtain the following proposition, which follows easily from \cite{San06}[Proposition 3.2.12]
    \begin{proposition}
        Let $m = (1,\mu) \in C$, i.e. $|\mu| < A$. Then, any semiclassical solution of $hP_1 u = u, \ hP_2 u = \mu u$ is an oscillatory function associated to the Lagrangian manifold $L_\mu = L_{1,\mu}$ in the sense of Appendix \ref{app:lagr}. Furthermore, for any $\eps > 0$, all constants are uniformly bounded as long as $|\mu| < (1 - \eps)A$.
    \end{proposition}
    
    \myindent The proof uses a microlocal normal form reducing the problem to the case $hP_j = \frac{h}{j} \partial_j$.
    
    \myindent Now, this proposition applies to the joint eigenfunctions $\phi_{k,n}$, thanks to the following lemma.
    
    \begin{lemma}
      There holds $Z \subset C$
    \end{lemma}
    
    \begin{proof}
        Observe that $\lambda_{k,n}^2 = \int |\nabla \phi_{k,n}|^2$ is the Dirichlet integral of $\phi_{k,n}$. Now, for any solution of $P_2 u = nu$, the Dirichlet integral of $u$ is given by $Q(u) = \int_{\mathbb{S}} |\partial_s u|^2 + \frac{n^2}{a^2(s)} |u|^2$ which satisfies $Q(u) > \frac{n^2}{A} \|u\|_{L^2}^2$ for every nonzero $u$.
    \end{proof}
    
    \subsubsection{Fold-type caustics and uniform BKW estimate of joint eigenfunctions}\label{sec:caustics}
    
    \myindent In this section, we will prove the following  estimate, which is standard in the context of fold-type caustics, see \cite{Gu-Sc}.
    
\begin{proposition}\label{airyestimate}
For all $x \in K\subset U$, for all $\nu \in Z$, there holds
\[
    |(\phi_{\nu})(x)| \lesssim_K   |\nu|^{\frac{1}{6}} <d(x,\mathcal{C}_{\nu}) |\nu|^{\frac{2}{3}}>^{-\frac{1}{4}},
\]
where $\mathcal{C}_{\nu} = \{(s,\theta): \ a(s) = \nu,\ \theta \in S^1\}$ is the caustic set of the Lagrangian torus
$        L_{\nu} := p^{-1}(\nu) $    
and $d(x,\mathcal{C}_{\nu})$ is the distance between the point $x$ and the set $\mathcal{C}_{\nu}$.
\end{proposition}
\myindent Note that it follows from Lemma \ref{prop:Z}, that $L_\nu $ is a Lagrangian torus. Note also that the caustic set of $L_\nu $ is the same as the caustic
set of $L_{\tau \nu} $ for any $\tau >0 $. 

\myindent The result will essentially follow from \cite{Gu-Sc}. Their  main result is 
\begin{theorem}
Near any point of a fold singularity of a Lagrangian manifold $L$, any oscillating function $u(x,h)$ of order $0$
admits the following behaviour:
\[ u(x,h) = e^{i\rho_0 (x)/h} \left\{h^{-1/6}a_0(x,h){\rm Ai} \left( \frac{\rho_1(x)}{h^{2/3}}\right)a_0 (x,h)+
h^{1/6}a_1(x,h){\rm Ai}' \left( \frac{\rho_1(x)}{h^{2/3}}\right)a_1 (x,h)  \right\} \]
where
\begin{itemize}
  \item 
$a_0$ and $a_1$ admit asymptotic  expansions in $h$ of degree $0$. 
\item 
   $\rho _0$ and $  \rho_1$ are smooth real real-valued functions.
\item  The caustic is defined by $\rho_1=0$  and $d\rho_1 \ne 0$ at each point of the caustic.
  \end{itemize} 
\end{theorem}

\myindent We then use the bound  for Airy 
${\rm Ai}(t)  =O\left( <t>^{-1/4} \right) $ and ${\rm Ai}'(t)=O\left(  <t>^{1/4}\right) $.
The contribution of ${\rm Ai}$ gives the right bound in the proposition, while that of ${\rm Ai}'$ is $O(1)$ (see Appendix \ref{sec:airy}). 

\myindent But we have to take care of the caustics approaching the equator or the poles.
In fact, this is not at all a problem because $x$ stays in some compact $K\subset V$:
if $\nu $ is close to $\partial C$, then $\phi_\nu (x)$ has fast decay in $\nu $.
If $\nu $ is close to the axis $v=0$, $x$ is at a distance $O(1)$ of the corresponding caustic, hence $\phi_\nu (x)$ stays bounded. 


\subsubsection{Distribution of the caustic sets}\label{sec:fol}

\myindent In order to apply Proposition \ref{airyestimate}, we need to prove estimates on the distribution of the caustic sets $\mathcal{C}_\mu$ for 
$\mu \in \mathcal{I}=]-A,A[\setminus 0$ over $U$.

\myindent We will first give a formulation in the context of {\it Differential Topology}. Let $W$ be a $3D$ manifold,
          and $p : W \to \R$ a smooth function whose differential does not vanish.
          For any $\mu \in p(W)$, let $L_\mu  \subset W$ be the $2D$ surface $p^{-1}(\mu)$. Let us give a $2D$ manifold $U$ and a submersion $\Pi : W \to U$.
          Our main assumption is that the restrictions $\Pi_\mu$ of $\Pi$ to $L_\mu$ admits fold singularities along smooth curves $\Gamma_\mu \subset L_\mu$.
          We denote by $C_\mu$ the caustics, namely $C_\mu = \Pi(\Gamma_\mu)$. We assume that the curves $C_\mu$
          are smooth submanifolds of $M$. Then, the following holds.

\begin{lemma}\label{lemm:caus}
    The family of curves $C_\mu$ is a smooth foliation of $\Pi(W)$ and the differential of the function $J$ defined by $J_{|C_\mu} := \mu$ does not vanish.
\end{lemma}

\begin{proof}[Proof of Lemma \ref{lemm:caus}]
  Let $Y \subset W$ be the set defined by $Y = \cup_\mu \Gamma_\mu$.
  Let us show that $Y$ is a surface foliated by the $\Gamma_\mu$. The $\Gamma_\mu$ form a smooth family,
  meaning that there exists a smooth function $F(s,\mu)$ so that $\Gamma_\mu = \{F(\cdot,\mu)\}$.
  It is enough to check that the differential of the function $K:= p_{|Y}$ does not vanish. Let $\gamma(\mu) = F(t,\mu)$ for a $t$ fixed.
  Then, $dp(\dot{\gamma}) = 1 = dK(\dot{\gamma}) \neq 0$. This proves also that $F$ is an embedding.
    
  \myindent It is then enough to prove that $\Pi_{|Y}$ is a diffeomorphism. This follows from the fact that $\Pi$ is a submersion: the kernel of $d\Pi$
  along $Y$ is transversal to $\Gamma_\mu$. For more details, we refer to \cite{Go-Gu}[Chapter III, section 4].
\end{proof}

\myindent Now, we can apply this Lemma to $W := \Omega \cap p_1^{-1}(1)$, $p=p_2$  and $\Pi:W \rightarrow U$  the bundle projection. We  deduce the following:

\begin{corollary}\label{spacingcor}
    For all $\mu, \mu' \in \mathcal{K} \subset \subset  \mathcal{I}$, there holds
    \[
    d(\mathcal{C}_\mu\cap U, \mathcal{C}_{\mu'} \cap U) \gtrsim |\mu -\mu'|.
    \]
\end{corollary}

\subsection{Proof of Theorem \ref{main} in the case of the sphere}


  \myindent Using Equation (\ref{equ:id}) and Proposition \ref{airyestimate},  we can start with the  bound
    \begin{equation}\label{fstboundproof}
      \left\| P_{\lambda,\lambda^{-\frac{1}{3}}}\right\|_{L^2(M) \to L^{\infty}(K)}^2 \lesssim
      \sup_{x \in K} \sum_{\nu=(\lambda_{k,n},n) \in Z, |\lambda_{k,n} - \lambda| \leq \lambda^{-1/3}}
      |\nu|^{\frac{1}{3}} <d(x,\mathcal{C}_{\nu}) |\nu|^{\frac{2}{3}}>^{-\frac{1}{2}}.
    \end{equation}
    
    \myindent Now,
    we know that the number of nonzero terms in the sum is a $O(\lambda^{\frac{2}{3}})$ thanks to the improved estimate on the remainder of Weyl counting function \eqref{generic}.
    Moreover, let us enumerate those $\nu \in Z$ such that $|\lambda_{k,n} - \lambda| \leq \lambda^{-1/3}$,  say by $\nu_1,\cdots, \nu_l=(\lambda_l, n_l),\cdots , \nu_K$, where $K = O(\lambda^{\frac{2}{3}})$. Then, if we denote
    \[ \mu_l :=|n_l|/\lambda _l  ~, \]
    it follows from Lemma \ref{gaplem}, that there holds \[ |\mu_l - \mu_{l'}| \gtrsim |l - l' |/ \lambda.\] 
   
    \myindent In particular, thanks to Corollary \ref{spacingcor}, we may bound the right-hand side of \eqref{fstboundproof} by
    \[
    \begin{split}
      \sum_{l = 0}^{O(\lambda^{\frac{2}{3}})} \lambda^{\frac{1}{3}} <l\lambda^{-1} \lambda^{\frac{2}{3}}>^{-\frac{1}{2}}
      &= \lambda^{\frac{1}{3}} \sum_{l = 0}^{O(\lambda^{\frac{2}{3}})} <l \lambda^{-\frac{1}{3}}>^{-\frac{1}{
        2}}\\
        &= \lambda^{\frac{1}{3}}\sum_{l = 0}^{\lambda^{\frac{1}{3}}} O(1) + \lambda^{\frac{1}{3}}\sum_{l = \lambda^{\frac{1}{3}}}^{O(\lambda^{\frac{2}{3}})} \lambda^{\frac{1}{6}} l^{-\frac{1}{2}}.
    \end{split}
    \]
    
     \myindent Now, the first sum is of order $O(\lambda^{\frac{2}{3}})$. The second may be bounded by
    \[
      \lambda^{\frac{1}{2}} \sum_{l = 0}^{O(\lambda^{\frac{2}{3}})} l^{-\frac{1}{2}} \lesssim \lambda^{\frac{1}{2}}
      \left(\lambda^{\frac{2}{3}}\right)^{\frac{1}{2}} \lesssim \lambda^{1 - \frac{1}{6}},
    \]
    which concludes the proof of the theorem in the case of spheres of revolution. 

\section{The case of the Euclidean disk} \label{sec:eucl}

 \myindent Let us consider the Euclidean disk $\mathbb{D}$ with Dirichlet boundary conditions\footnote{
 We could have looked also at the Neumann boundary conditions with similar
techniques and results}.
 The proof for the case where $K\cap \partial \D=\emptyset $ is quasi identical to
the case of the sphere, but some estimates  are  simpler
 using the explicit expression of the eigenfunctions in terms of the Bessel
functions $J_n$. 
 
\myindent As it is well known and used in
\cite{YCdV-disque, CGJ17}, the normalized  eigenfunctions are given in polar
coordinates by
\[ \phi_{k,n}(r,\theta) = c_{k,n} J_n \left( \lambda_{k,n} r\right)e^{in\theta } \]
for some positive  scalars $c_{k,n}$, 
where
\[\forall n \in \Z,~  J_n (t)= \frac{1}{2\pi} \int _{\R/2\pi \Z} e^{i(n\alpha -
t\sin \alpha) } d\alpha   ~, \]
  $\lambda _{k,n} $ is for $k\in \N$ the $k-$th zero of $J_n$.
We have $\Delta \phi_{k,n}= -\lambda_{k,n}^2 \phi_{k,n} $.
In what follows, we restrict ourselves to $n\geq 0$ using $J_{-n}(t)=(-1)^n{J_n }(t)$. We may moreover restrict to $\eps \leq \frac{n}{\lambda_{k,n}}$ for some $\eps > 0$ small enough, since from the explicit expression, $J_n$ decreases exponentially fast for $t > \frac{n}{\lambda_{k,n}}$, and $O \notin K$.

\myindent For $n\geq \epsilon \lambda_{k,n}$ with some $\epsilon>0$, the Lagrangian manifolds
$L_\mu $ are given by 
\begin{equation} \label{equ:lag-d} L_\mu  := \left\{ \left(r,\theta, \pm
\sqrt{\mu^{-2}-r^{-2}}, 1 \right)~ | ~r\geq \mu \right\},
\end{equation}
with $\mu =\mu_{k,n}= \frac{n}{\lambda_{k,n}}$ and the semi-classical parameter $1/n
$.

\myindent Using directly the previous expression of the eigenfunctions, we are left to
evaluate the following quantity:
\[ \|1_{[\lambda-\lambda^{-1/3},\lambda + \lambda^{-1/3}]}(\sqrt{-\Delta})
\|_{L^2 (\D) \rightarrow L^\infty (K)}=\sqrt{
\sup _{x\in K} \sum_{|\lambda_{k,n} - \lambda|\leq \lambda^{-1/3}} \phi_{k,n}(x)|^2}. \]

\myindent The joint spectrum is now the set  $Z:=\{(\lambda _{k,n}, n )|k\in \N , n\in \Z \}$
where the $ \lambda _{k,n}$'s  are the zeroes of $J_n$.
Note that $Z\subset \{ (\lambda, n) ~|~ |n| \leq \lambda \} $ as follows from the
book \cite{Watson}, p. 86.

\myindent Note that we really need to avoid $O$ in the compact sets $K$ in order to get non
trivial results:
indeed, look at  the sequence $\phi_k (r,\theta )=J_0 (\lambda_{k,0} r) $ where
$\lambda_{k,0}$ is the sequence of zeroes of $J_0$.
We have then $\phi_k(O)=1 $ while $\| \phi_k \|_{L^2(\D)}\simeq
1/\sqrt{\lambda_{k,0}}$.

\quad

\myindent We will need the following lemma
\begin{lemma}\label{lemm:l2}
  The $L^2(\D) $ norms of the eigenfunctions
$J_n\left((\lambda_{k,n}r\right))e^{in\theta }$ for
  $\mu_{k,n}= |n|/\lambda _{k,n}  $  is equivalent to $({1-\mu})^{1/4}/\sqrt{n}$
uniformly as $n\rightarrow \infty $.  
\end{lemma}
{\it Proof of Lemma.--}
The proof follows simply from the asymptotics
\[ J_n (\lambda   r ) \sim n^{-1/3}{\rm Ai} \left(-n^{2/3} \rho (r) \right) \]
with $\rho (r) \sim r-\mu  $ which is known  from  Section \ref{sec:caustics}.
We need also 
 the asymptotics
 $\int_{-X}^\infty {\rm Ai}^2 (X) dX \sim \sqrt{X}$ as $X \rightarrow +\infty$.\hfil $\square $

\begin{remark}
    Observe that, for any term in the sum, there holds
    \[ 1 - \mu_{k,n} \gtrsim \lambda^{-\frac{2}{3}}\]

   This follows from Appendix \ref{app:bessel}.
\end{remark}

\myindent The following property follows: 
 \begin{proposition} \label{prop:norm}  Given $0< \epsilon <1  $, The constants
$c_{k,n}$ satisfy, for $\epsilon \lambda_{k,n} \leq |n| < \lambda_{k,n}$,
   $c_{k,n}\sim n^{1/2} (1-\mu_{k,n}) ^{-1/4} $ with
   $\mu_{k,n} = |n|/ \lambda_{k,n} $ uniformly as $\lambda_{k,n} \rightarrow \infty $.
   \end{proposition} 

\subsection{$K\cap \partial\D =\emptyset$}
 
\myindent Now, we can apply the same formalism than in the general case.
Indeed, the $L_\mu $ admit fold singularities with caustic sets $C_\mu = \{ r=\mu \}$. Moreover, the spacing of the zeroes of $J_n$ which are bounded below
asymptotically for $n$ large, ensure that, for each value of $n$, there is at most one value of $k$ such that $|\lambda_{k,n} - \lambda|\leq \lambda^{-1/3}$. As a consequence, the caustics appearing in the sum are spaced by at least $\lambda^{-1}$ one from another as was the case for the sphere.   

\subsection{$K\cap \partial \D \ne \emptyset$ }

\myindent The issue when $K$ can meet the boundary is that the constants $\mu_{k,n}$ become very close to $1$, hence the previous method doesn't work anymore. The idea to fix this is that we can actually prove that, the closer the caustics are to the boundary, the more they are spaced one from another.

\myindent We recall the set of parameters:
\begin{itemize}
\item $\nu_{k,n}=(\lambda_{k,n},n )\in Z $. We will assume that $x=(r,\theta)$ lies
near the boundary, i.e.  $x\in \D_c=\{ r\geq c>0 \}$
  and we can assume that $\nu_{k,n} $ belongs to some cone $Y:=\{ (u,v)~|~c_1u \leq
v < u \} $ with $c_1>0$: the other eigenfunctions are indeed uniformly bounded
  in $\D_c$.
\item The parameter $\mu =\mu_{k,n}= n/\lambda_{k,n} \in [c_1,1[ $ is associated to
a Lagrangian manifold $L_\mu $ defined by Equation
    (\ref{equ:lag-d}) and to the corresponding caustic set $C_\mu=\{ r=\mu \} $.
\end{itemize}

\myindent We have the
\begin{lemma} \label{lemm:zeroes}
  There holds, for $\nu_{k,n}\in Y $, 
  \[ \lambda_{k,n} =n+ nF\left( \frac{a_k}{n^{2/3}}\right) + O \left( \frac{1}{n}
\right) \]
  with $F$ and $ a_k $ defined in Appendix \ref{app:bessel} and
  $ \frac{a_k}{n^{2/3}}$ bounded.
\end{lemma}
\myindent Indeed, we have
$n/\lambda _{k,n} =1/(1+ F\left( \frac{a_k}{n^{2/3}}\right)) \geq c_1 $. This
implies that the argument in $F$ is bounded because $F(t)\rightarrow +\infty $
as $t\rightarrow +\infty $. 
 
 \quad

\myindent Now, we can prove the following lemma.
\begin{lemma} \label{lemm:ineg}
  Let us consider, for $j=1,2$,  $\nu_j=(\lambda_{k_j,n_j},n_j) \in B_\lambda \cap Z
\cap Y$.
  Then we have, for $\lambda $ large enough,  $a_{k_1}\ne a_{k_2}$
  and, assuming $n_1 < n_2$:  
\[ n_1^{1/3} |a_{k_1} -a_{k_2} | \lesssim n_2-n_1 \]
\end{lemma}
\begin{proof}
    The first assertion follows from the fact that if $G_a$ is defined by $G_a(x)=x+
xF(a/x^{2/3}) $, $G'_a \geq 1 $. 
    We have then
    \[ \left( n_1+n_1F\left( \frac{a_{k_1}}{n_1^{2/3}}\right)\right) -\left(
n_2+n_2F\left( \frac{a_{k_2}}{n_2^{2/3}}\right)\right)= O\left(
    \lambda^{-1/3 }\right) \]
    \myindent Let us denote $F_j=F\left( \frac{a_{k_j}}{n_j^{2/3}}\right)$. We first get
    \begin{equation}\label{intermediaterandomeq}
    n_2 |F_2 -F_1| \lesssim n_2-n_1.
    \end{equation}
    \myindent Then, using the fact that $F'\geq 1$, 
    \[ |F_2-F_1| \geq  \left|\frac{a_{k_1}}{n_1^{2/3}} -\frac{a_{k_2}}{n_2^{2/3}}  
\right| \]
    \myindent We rewrite
    \begin{equation}\label{randomeq}
     \frac{a_{k_1}}{n_1^{2/3}} -\frac{a_{k_2}}{n_2^{2/3}}=
\frac{a_{k_1}-a_{k_2}}{n_1^{2/3}}+
a_{k_2}\left(\frac{1}{n_1^{2/3}}-\frac{1}{n_2^{2/3}}  \right) 
\end{equation}
    \myindent We remark finally, using the fact that $n_1\sim n_2$,  that
    \[ n_2\left| a_{k_2}\left(\frac{1}{n_1^{2/3}}-\frac{1}{n_2^{2/3}}\right) \right|
\lesssim n_2\left( |a_{k_2}| n_1^{-5/3}  (n_2-n_1)\right) \lesssim n_2^{-\frac{2}{3}}|a_{k_2}| | n_2- n_1|,~~~ 
\]
  which is very small compared to $(n_2 - n_1)$ if we are close enough to the boundary. Hence we may put the second term of the right-hand side of \eqref{randomeq} to the right-hand side of \eqref{intermediaterandomeq}, which yields the result. 
\end{proof}

\myindent Now, since the distance between two caustics is given by
\begin{equation}
    \frac{n_1}{\lambda_{k_1,n_1}} - \frac{n_2}{\lambda_{k_2,n_2}} = \frac{n_1 - n_2}{\lambda} + O(\lambda^{-\frac{4}{3}}),
\end{equation}
this lemma yields an improvement on the spacing between caustics. Indeed, find $\mu = 1 - \eps$, for some $\eps > \lambda^{-\frac{2}{3}}$, and write it as $\eps = \lambda^{-\frac{2}{3}} k^{\frac{2}{3}}$ for some $k \geq 1$. Then, the lemma proves that, in an interval say $[1 -\frac{1}{2}\eps, 1 - 2\eps]$, the caustics are spaced by $\lambda^{-1} \eps^{-\frac{1}{2}}$.

\quad

\myindent Now, let us label $\nu_j =(\lambda_j, n_j ) , ~j=1,\cdots , k $, $k = O(\lambda^{\frac{2}{3}})$, the  set $Z
\cap Y \cap B_\lambda $ with $n_k< \cdots <n_1 < \lambda $. We write $$\mu_j = n_j/\lambda  $$. Then, we need to bound, for $\in [c,1[$, the sum
\[ S_{\lambda} := \sum_{j=1}^k |\phi_{j}(x)|^2. \]

\myindent Now, for any $j$ such that $r > \mu_j$, there holds from the BKW decay
\begin{equation}
    |\phi_j(r)|^2 \lesssim (1 - \mu_j)^{-\frac{1}{2}} \lambda^{\frac{1}{3}} <|r - \mu_j|\lambda^{\frac{2}{3}} >^{-\frac{1}{2}},
\end{equation}
while, for any $N\geq 1$, and for any $j$ such that $r < \mu_j$, the exponential decay of the Airy function on the right half-line ensures that
\begin{equation}
    |\phi_j(r)|^2 \lesssim (1 - \mu_j)^{-\frac{1}{2}} \lambda^{\frac{1}{3}} <|r - \mu_j|\lambda^{\frac{2}{3}} >^{-N}.
\end{equation}

\myindent Hence, introducing a cutoff between those two cases, it is natural to decompose the set into three parts. Writing $r = 1 - \eta$, we introduce, for some $0< \alpha << 1$,
\begin{equation}
\begin{split}
    &A:= \{j \ \text{such that}\ \mu_j -r > \lambda^{-\frac{2}{3} + \alpha} \}\\
    &B :=\{j \ \text{such that} \ -\lambda^{-\frac{2}{3} + \alpha} <r-\mu_j < \max(2\eta, \lambda^{-\frac{2}{3}+\alpha})\}\\
    &C:= \{j \ \text{such that}\ r-\mu_j > \max(2\eta, \lambda^{-\frac{2}{3}+\alpha})\}.
\end{split}
\end{equation}

\quad

\myindent First, for the sum on $A$, we are in the regime of exponential decay, so we may write, for any $N \geq 1$,
\[\begin{split}
    \sum_A |\phi_j(r)|^2 &\leq \sum_A (1 - \mu_j)^{-\frac{1}{2}} \lambda^{\frac{1}{3}} <|r - \mu_j|\lambda^{\frac{2}{3}}>^{-N} \\
    &\leq |A| (\lambda^{-\frac{2}{3}})^{-\frac{1}{2}} \lambda^{\frac{1}{3}} \lambda^{-N \alpha} \\
    &= O(\lambda^{-\infty}),
\end{split}\]
since we may choose $N$ arbitrarily large (observe that $\alpha$ is fixed before $N$).

\myindent With regards to the sum on $B$, the point is that we don't have decay coming from the Airy functions. However, we expect the set $B$ itself to be relatively small. Assume first that $\eta \lesssim \lambda^{-\frac{2}{3}+\alpha}$. Then, for the points in $B$, the caustics are spaced by at least $\lambda^{-\frac{2}{3}}$, so we easily get that the cardinal of $B$ is bounded by $\lambda^{\alpha}$. Hence, overall, there holds
\[\sum_B |\phi_j(r)|^2 \lesssim \lambda^{\alpha} (\lambda^{-\frac{2}{3}})^{-\frac{1}{2}} \lambda^{\frac{1}{3}} = \lambda^{\frac{2}{3} + \alpha}.\]

\myindent Assume now that $\eta  \gtrsim \lambda^{-\frac{2}{3}+\alpha}$. Observe that we need only bound the sum on those $j$ such that $\mu_j \in [r - \lambda^{-\frac{2}{3} + \alpha}, r - 2\eta]$. Then, for the points in $B$, the caustics are spaced at least by $\lambda^{-1}\eta^{-\frac{1}{2}}$. In particular, the cardinal of $B$ is bounded by $\lambda \eta^{\frac{3}{2}}$, and by $O(\lambda^{\frac{2}{3}})$. Hence, we may bound
\[
\begin{split}
    \sum_{\mu_j \in [r - \lambda^{-\frac{2}{3} + \alpha}, r - 2\eta]} |\phi_j(r)|^2 &\lesssim \sum_{l =1}^{\min(\lambda \eta^{\frac{3}{2}},O(\lambda^{\frac{2}{3}}))} \eta^{-\frac{1}{2}} \lambda^{\frac{1}{3}} < l \lambda^{-1} \eta^{-\frac{1}{2}} \lambda^{\frac{2}{3}}>^{-\frac{1}{2}} \\
    &\lesssim \lambda^{\frac{1}{2}} \eta^{-\frac{1}{4}} \sum_{l = 1}^{\min(\lambda \eta^{\frac{3}{2}},O(\lambda^{\frac{2}{3}}))} l^{-\frac{1}{2}} \\
    &\lesssim \eta^{-\frac{1}{4}} \lambda^{\frac{1}{2}} \min(\lambda^{\frac{1}{2}} \eta^{\frac{3}{4}}, \lambda^{\frac{1}{3}})\\
    &\lesssim \min(\lambda \eta^{\frac{1}{2}}, \lambda^{\frac{5}{6}} \eta^{-\frac{1}{4}})\\
    &\lesssim \lambda^{\frac{8}{9}}.
\end{split}\]

\myindent Finally, for the sum on $C$, observe that $|r - \mu_j| \simeq 1 - \mu_j$, so we may bound it by
\[\sum_{j \in C} (1-\mu_j)^{-1}\]

\myindent Now, since, using the asymptotics of the zeroes of the Airy function, there holds
\begin{equation}
    \mu_j \simeq \lambda^{-\frac{2}{3}} k_j^{\frac{2}{3}},
\end{equation}
thus
\[\sum_{j \in C} (1-\mu_j)^{-1} \lesssim \lambda^{\frac{2}{3}} \sum_{l=1}^{O(\lambda^{\frac{2}{3}})} l^{-\frac{2}{3}} \lesssim \lambda^{\frac{8}{9}},\]
which concludes the proof.


\vfil \eject

 {\bf \huge Appendices} 

 \appendix
   \section{Airy function}\label{sec:airy}
   \myindent The Airy function is defined by the following integral formula:
   \[ {\rm Ai}(x):= \frac{1}{\pi} \int _\R e^{i\left( \frac{t^3}{3} + tx \right)}dt \]
   \myindent The Airy function is smooth and 
   satisfies the following asymptotics:
   if $x>>1$, 
   $ {\rm Ai}(x) =O \left(x^{-\infty} \right)$ and similarly for ${\rm Ai}'$; 
   if $x<<-1$,
   \[ {\rm Ai}(x)\sim  \frac{1}{\sqrt{\pi x^{1/4}}} \cos \left( \frac{2}{3}x^{3/2} \right)=O \left(x^{-1/4}\right)\]
     while
     \[ {\rm Ai}'(x)=O \left(x^{1/4}\right)~.\]

   \section{Lagrangian oscillatory integrals}\label{app:lagr}

   \myindent Let $L\subset T^\star M $ be Lagrangian manifold. It is known  from the work of H\"ormander, that $L$ can be locally described using a
   phase function $S(x,\theta )$ with $x\in  M$ and $\theta \in \R^N $:
   under some non degeneracy assumption on $S$, we have
   \[ L:=\{ (x,d_xS)~|~d_\theta S=0 \} \]
   and the oscillatory integrals associated to $L$ of order $0$ are defined by
   \[ u(x,h)=\frac{1}{(2\pi h)^{N/2} }\int_{\R^N}e^{iS(x,\theta ) /h} a(x,\theta ) d\theta \]
        where $a$ is a symbol of degree $0$ in $h$ compactly supported in $x$.

        \myindent Outside the caustic point, $u$ is an ordinary WKB function:
        \[ u(x) =e^{iS_1(x)/h}a_1(x,h) \]
        where $a$ is a symbol of degree $0$ in $h$ compactly supported in $x$ and $L$ is the graph of $dS_1$. 

        \section{Zeroes of $J_n$}\label{app:bessel}

        \myindent In the paper \cite{Ol54}, the following uniform asymptotic expansion for the $k$th zeroes of $J_n$
        is
        \[ \lambda_{k,n}= n p_0 \left( \frac{a_k}{n^{2/3}} \right) +O(1/n) \]
        with $p_0\in C^\infty (\R^+, \R^+)$, a smooth function, and $a_k $ the $k$th zero of Ai.
        The function $p_0$ satisfies
        $p_0(0)=1$, $\forall t\geq 0,~ p'(t)\geq 1 $.


\begin{thebibliography}{CdVGJ17}

\bibitem[BL25]{Bam25}
Dario Bambusi and Beatrice Langella.
\newblock Globally Integrable Quantum Systems and Their Perturbations.
\newblock{\em In: Cassano, B., Cunden, F.D., Gallone, M., Ligabò, M., Michelangeli, A. (eds) Singularities, Asymptotics, and Limiting Models. Springer INdAM Series, vol 64. 2025.}

\bibitem[BD15]{bourgain2015proof}
Jean Bourgain and Ciprian Demeter.
\newblock The proof of the {$L^2$} decoupling conjecture.
\newblock {\em Annals of mathematics}, pages 351--389, 2015.

\bibitem[B{\'e}r77]{berard1977wave}
Pierre B{\'e}rard.
\newblock On the wave equation on a compact riemannian manifold without
  conjugate points.
\newblock {\em Mathematische Zeitschrift}, 155(3):249--276, 1977.

\bibitem[BK17]{buttcane2017fourth}
Jack Buttcane and Rizwanur Khan.
\newblock On the fourth moment of hecke--maass forms and the random wave
  conjecture.
\newblock {\em Compositio Mathematica}, 153(7):1479--1511, 2017.

\bibitem[Bou93a]{bourgain1993eigenfunctiona}
Jean Bourgain.
\newblock {\em Eigenfunction bounds for compact manifolds with integrable
  geodesic flow}.
\newblock Institut des Hautes Etudes Scientifique, 1993.

\bibitem[Bou93b]{bourgain1993eigenfunction}
Jean Bourgain.
\newblock Eigenfunction bounds for the laplacian on the n-torus.
\newblock {\em International Mathematics Research Notices}, 1993(3):61--66,
  1993.

\bibitem[Bou13]{bourgain2013moment}
Jean Bourgain.
\newblock Moment inequalities for trigonometric polynomials with spectrum in
  curved hypersurfaces.
\newblock {\em Israel Journal of Mathematics}, 193(1):441--458, 2013.

\bibitem[BS17]{blair2017refined}
Matthew Blair and Christopher Sogge.
\newblock Refined and microlocal kakeya--nikodym bounds of eigenfunctions in
  higher dimensions.
\newblock {\em Communications in Mathematical Physics}, 356:501--533, 2017.

\bibitem[BS18]{blair2018concerning}
Matthew Blair and Christopher Sogge.
\newblock Concerning toponogov’s theorem and logarithmic improvement of
  estimates of eigenfunctions.
\newblock {\em Journal of Differential Geometry}, 109(2):189--221, 2018.

\bibitem[BS19]{blair2019logarithmic}
Matthew Blair and Christopher Sogge.
\newblock Logarithmic improvements in {$L^p$ to $ L^p$} bounds for
  eigenfunctions at the critical exponent in the presence of nonpositive
  curvature.
\newblock {\em Inventiones mathematicae}, 217:703--748, 2019.

\bibitem[CdV77a]{CdV-lattice}
Yves Colin~de Verdi{\`e}re.
\newblock Nombre de points entiers dans une famille homoth{\'e}tique de
  domaines de {$R^n$}.
\newblock {\em Ann. Sci. {\'E}c. Norm. Sup{\'e}r. (4)}, 10:559--576, 1977.

\bibitem[CdV77b]{de1977quasi}
Yves Colin~de Verdi{\`e}re.
\newblock Quasi-modes sur les vari{\'e}t{\'e}s riemanniennes.
\newblock {\em Inventiones mathematicae}, 43(1):15--52, 1977.

\bibitem[CdV80a]{YCdV-sp-joint-II}
Yves Colin~de Verdi{\`e}re.
\newblock Spectre conjoint d'op{\'e}rateurs pseudo-diff{\'e}rentiels qui
  commutent {II}. {Le} cas int{\'e}grable.
\newblock {\em Math. Z.}, 171:51--73, 1980.

\bibitem[CdV80b]{de1980spectre}
Yves Colin~de Verdi{\`e}re.
\newblock Spectre conjoint d'op{\'e}rateurs pseudo-diff{\'e}rentiels qui
  commutent: Ii. le cas int{\'e}grable.
\newblock {\em Mathematische Zeitschrift}, 171:51--73, 1980.

\bibitem[CdV11]{YCdV-disque}
Yves Colin~de Verdi{\`e}re.
\newblock On the remainder in the {Weyl} formula for the {Euclidean} disk.
\newblock In {\em Actes de S\'eminaire de Th\'eorie Spectrale et G\'eom\'etrie.
  Ann\'ee 2010--2011}, pages 1--13. St. Martin d'H{\`e}res: Universit{\'e} de
  Grenoble I, Institut Fourier, 2011.

\bibitem[CdVGJ17]{CGJ17}
Yves Colin~de Verdi{\`e}re, Victor Guillemin, and David Jerison.
\newblock Singularities of the wave trace for the {Friedlander} model.
\newblock {\em Journal d'Analyse Math{\'e}matique}, 133:1--25, 2017.

\bibitem[CG19]{canzani2019growth}
Yaiza Canzani and Jeffrey Galkowski.
\newblock On the growth of eigenfunction averages: Microlocalization and
  geometry.
\newblock {\em Duke Mathematical Journal}, 168(16):2991--3055, 2019.

\bibitem[CG21]{canzani2021eigenfunction}
Yaiza Canzani and Jeffrey Galkowski.
\newblock Eigenfunction concentration via geodesic beams.
\newblock {\em Journal f{\"u}r die reine und angewandte Mathematik (Crelles
  Journal)}, 2021(775):197--257, 2021.

\bibitem[Cha25]{chabert2025bounds}
Ambre Chabert.
\newblock Bounds for quasimodes with polynomially narrow bandwidth on surfaces
  of revolution.
\newblock {\em arXiv preprint arXiv:2502.00143}, 2025.

\bibitem[CL26]{Chabert26}
Ambre Chabert and Thibault Lefeuvre.
\newblock {\em Improved $L^\infty$ bounds for eigenfunctions of magnetic Laplacians on hyperbolic surfaces.}
\newblock To appear.

\bibitem[DG24]{demeter2024l2}
Ciprian Demeter and Pierre Germain.
\newblock {$L^2$} to {$L^p$} bounds for spectral projectors on the euclidean
  two-dimensional torus.
\newblock {\em Proceedings of the Edinburgh Mathematical Society},
  67(2):431--459, 2024.

\bibitem[Dui74]{duistermaat1974oscillatory}
Johannes Duistermaat.
\newblock Oscillatory integrals, lagrange immersions and unfolding of
  singularities.
\newblock {\em Communications on Pure and Applied Mathematics}, 27(2):207--281,
  1974.

\bibitem[EGK24]{eswarathasan2024pointwise}
Suresh Eswarathasan, Allan Greenleaf, and Blake Keeler.
\newblock Pointwise weyl laws for quantum completely integrable systems.
\newblock {\em arXiv preprint arXiv:2411.10401}, 2024.

\bibitem[Ger23]{germain2023l2}
Pierre Germain.
\newblock {$L^2$} to {$L^p$} bounds for spectral projectors on thin intervals
  in riemannian manifolds.
\newblock {\em arXiv preprint arXiv:2306.16981}, 2023.

\bibitem[GG80]{Go-Gu}
Marty Golubitsky and Victor Guillemin.
\newblock {\em Stable mappings and their singularities. 2nd corr. printing},
  volume~14 of {\em Grad. Texts Math.}
\newblock Springer, Cham, 1980.

\bibitem[GM22]{germain2022boundsBook}
Pierre Germain and Simon L~Rydin Myerson.
\newblock Bounds for spectral projectors on tori.
\newblock In {\em Forum of Mathematics, Sigma}, volume~10, page e24. Cambridge
  University Press, 2022.

\bibitem[GRM22]{germain2022bounds}
Pierre Germain and Simon~L Rydin~Myerson.
\newblock Bounds for spectral projectors on generic tori.
\newblock {\em Mathematische Annalen}, pages 1--37, 2022.

\bibitem[GS73]{Gu-Sc}
Victor Guillemin and David Schaeffer.
\newblock Remarks on a paper of {D}. {Ludwig}.
\newblock {\em Bull. Am. Math. Soc.}, 79:382--385, 1973.

\bibitem[Hic20]{hickman2020uniform}
Jonathan Hickman.
\newblock Uniform {$L^p$} resolvent estimates on the torus.
\newblock {\em Mathematics Research Reports}, 1:31--45, 2020.

\bibitem[HK22]{humphries2022p}
Peter Humphries and Rizwanur Khan.
\newblock {$L^p$}-norm bounds for automorphic forms via spectral reciprocity.
\newblock {\em arXiv preprint arXiv:2208.05613}, 2022.

\bibitem[H{\"o}r68]{hormander1968spectral}
Lars H{\"o}rmander.
\newblock The spectral function of an elliptic operator.
\newblock {\em Acta Mathematica}, 121(1):193--218, 1968.

\bibitem[HR16]{hezari2016lp}
Hamid Hezari and Gabriel Rivi{\`e}re.
\newblock {$L^p$} norms, nodal sets, and quantum ergodicity.
\newblock {\em Advances in Mathematics}, 290:938--966, 2016.

\bibitem[HT15]{hassell2015improvement}
Andrew Hassell and Melissa Tacy.
\newblock Improvement of eigenfunction estimates on manifolds of nonpositive
  curvature.
\newblock In {\em Forum Mathematicum}, volume 27 (3), pages 1435--1451. De
  Gruyter, 2015.

\bibitem[Hum18]{humphries2018equidistribution}
Peter Humphries.
\newblock Equidistribution in shrinking sets and {$L^4$}-norm bounds for
  automorphic forms.
\newblock {\em Mathematische Annalen}, 371:1497--1543, 2018.

\bibitem[IS95]{iwaniec1995norms}
Henryk Iwaniec and Peter Sarnak.
\newblock {$L^{\infty}$} norms of eigenfunctions of arithmetic surfaces.
\newblock {\em Annals of Mathematics}, 141(2):301--320, 1995.

\bibitem[KF65]{kuznetsov1965asymptotic}
Nikolai~Vasil'evitch Kuznetsov and Boris~Vasil'evich Fedosov.
\newblock An asymptotic formula for eigenvalues of a circular membrane.
\newblock {\em Differentsial'nye Uravneniya}, 1(12):1682--1685, 1965.

\bibitem[Olv54]{Ol54}
Frank Olver.
\newblock The asymptotic expansion of bessel functions of large order.
\newblock {\em Phil. Trans. Royal Soc. London}, A247:328--368, 1954.

\bibitem[Sog88]{sogge1988concerning}
Christopher Sogge.
\newblock Concerning the {$L^p$} norm of spectral clusters for second-order
  elliptic operators on compact manifolds.
\newblock {\em Journal of functional analysis}, 77(1):123--138, 1988.

\bibitem[Sog01]{sogge2001riemannian}
Christopher Sogge.
\newblock Riemannian manifolds with maximal eigenfunction growth.
\newblock {\em S{\'e}minaire {\'E}quations aux d{\'e}riv{\'e}es partielles
  (Polytechnique) dit aussi" S{\'e}minaire Goulaouic-Schwartz"}, pages 1--16,
  2001.

\bibitem[STZ11]{sogge2011blowup}
Christopher Sogge, John Toth, and Steve Zelditch.
\newblock About the blowup of quasimodes on riemannian manifolds.
\newblock {\em Journal of Geometric Analysis}, 21(1):150--173, 2011.

\bibitem[SZ02]{sogge2002riemannian}
Christopher Sogge and Steve Zelditch.
\newblock Riemannian manifolds with maximal eigenfunction growth.
\newblock {\em Duke Mathematical Journal}, 114(3):387--437, 2002.

\bibitem[SZ16a]{sogge2016focal}
Christopher Sogge and Steve Zelditch.
\newblock Focal points and sup-norms of eigenfunctions.
\newblock {\em Revista Matem{\'a}tica Iberoamericana}, 32(3):971--994, 2016.

\bibitem[SZ16b]{sogge2016focalb}
Christopher Sogge and Steve Zelditch.
\newblock Focal points and sup-norms of eigenfunctions ii: the two-dimensional
  case.
\newblock {\em Revista matem{\'a}tica iberoamericana}, 32(3):995--999, 2016.

\bibitem[TZ02]{toth2002riemannian}
John Toth and Steve Zelditch.
\newblock Riemannian manifolds with uniformly bounded eigenfunctions.
\newblock {\em Duke Mathematical Journal}, 111(1):97--132, 2002.

\bibitem[TZ03a]{toth2003lp}
John Toth and Steve Zelditch.
\newblock {$L^p$} norms of eigenfunctions in the completely integrable case.
\newblock In {\em Annales Henri Poincare}, volume 4 (2), pages 343--368.
  Birkhauser Verlag Basel, 2003.

\bibitem[TZ03b]{toth2003norms}
John Toth and Steve Zelditch.
\newblock Norms of modes and quasi-modes revisited.
\newblock {\em Contemporary Mathematics}, 320:435--458, 2003.

\bibitem[VN06]{San06}
San V{\~u}~Ng{\d{o}}c.
\newblock {\em Syst{\`e}mes int{\'e}grables semi-classiques. {Du} local au
  global}, volume~22 of {\em Panoramas et Synth{\`e}ses}.
\newblock Paris: Soci{\'e}t{\'e} Math{\'e}matique de France (SMF), 2006.

\bibitem[Wat95]{Watson}
George~Neville Watson.
\newblock {\em A treatise on the theory of {Bessel} functions.}
\newblock Cambridge Univ. Press, 2nd edition, 1995.

\end{thebibliography}
\end{document}